\documentclass[secnum,leqno,11pt]{amsart}
\usepackage{amssymb, amsmath, amsthm}
\usepackage{graphics}
\usepackage{geometry}
\usepackage{indentfirst}
\usepackage[dvips]{graphicx}
\usepackage{eepic}
\usepackage{epic}
\newtheorem{thm}{Theorem}[section]
\newtheorem{cor}[thm]{Corollary}

\newtheorem{lmm}[thm]{Lemma}

\newtheorem{prp}[thm]{Proposition}

\theoremstyle{remark}
\newtheorem*{rmk*}{Remark}
\newtheorem*{ack}{Acknowledgement}

\makeatletter

\@addtoreset{equation}{section}
\makeatother
\begin{document}
\title[]{Polyharmonic weak Maass forms of\\ higher depth for $\mathrm{SL}_2(\mathbb{Z})$}
\author[]{Toshiki Matsusaka}
\address{Graduate School of Mathematics, Kyushu University, Motooka 744, Nishi-ku Fukuoka 819-0395, Japan}
\email{toshikimatsusaka@gmail.com}

\date{}

\begin{abstract}
The space of polyharmonic Maass forms was introduced by Lagarias-Rhoades, recently. They constructed its basis from the Taylor coefficients of the real analytic Eisenstein series. In this paper, we introduce polyharmonic weak Maass forms, that is, we relax the moderate growth condition at cusp, and we construct a basis as a generalization of Lagarias-Rhoades' works. As a corollary, we can obtain a preimage of an arbitrary polyharmonic weak Maass form under the $\xi$-operator.
\end{abstract}


\maketitle

\section{Introduction}\label{s1} 
We begin by discussing Kronecker's first limit formula. The Riemann zeta function $\zeta(s) := \sum_{n=1}^{\infty} n^{-s}$ has a simple pole at $s=1$, and the following Laurent expansion is known.
\begin{align*}
\zeta(s) = \frac{1}{s-1} + \gamma + O(s-1),
\end{align*}
where $\gamma = 0.577215\dots$ is Euler's constant. As an analogue we now consider the real analytic Eisenstein series given by
\begin{align*}
E(z,s) := \sum_{(m,n) \in \mathbb{Z}^2 \backslash (0,0)}\frac{y^s}{|mz+n|^{2s}}
\end{align*}
for $s \in \mathbb{C}$ with $\mathrm{Re}(s) > 1$ and a modular variable $z = x+iy$ in the upper half plane $\mathfrak{H}$. It is a non-holomorphic modular form of weight 0 on $\mathrm{SL}_2(\mathbb{Z})$, and meromorphically continued to the whole $s$-plane. Then this Eisenstein series has the Laurent expansion of the form
\begin{align*}
E(z,s) = \frac{\pi}{s-1} + 2\pi(\gamma - \mathrm{log}2 - \mathrm{log}(\sqrt{y}|\eta(z)|^2)) + O(s-1)
\end{align*}
where, for $q := e^{2\pi iz}$, $\eta(z) := q^{\frac{1}{24}}\prod_{n=1}^{\infty}(1-q^n)$ is the Dedekind eta function. This is so-called Kronecker's first limit formula. This limit formula has been extensively studied since a long time ago, and a lot of proofs are known, Kroneker's own and Shintani \cite{S} using the Barnes double Gamma function, and many others. Further results on Kronecker's first limit formula are reviewed in \cite{DIT3}. On the other hand, Lagarias-Rhoades \cite{LR} considered the higher Laurent coefficients of $E(z,s)$ from the viewpoint of harmonic Maass forms. In fact, $E(z,s)$ is an eigenfunction of the hyperbolic Laplacian $\Delta_0 := -y^2\bigl(\frac{\partial^2}{\partial x^2} + \frac{\partial^2}{\partial y^2}\bigr)$, and satisfies
\begin{align*}
\Delta_0 E(z,s) = s(1-s)E(z,s).
\end{align*}
Then they showed that the $r$-th Laurent coefficient $F_r(z)$ in the form
\begin{align*}
E(z,s) = \sum_{r=-1}^{\infty} F_r(z) (s-1)^r
\end{align*}
satisfies the differential equation $\Delta_0^{r+2}F_r(z) = 0$. Based on this property, they constructed a new space called polyharmonic Maass forms, and revealed the roles of these Laurent coefficients in this new space.\\

A harmonic Maass form of even weight $k \in 2\mathbb{Z}$ on $\mathrm{SL}_2(\mathbb{Z})$ is a smooth function $f$ on $\mathfrak{H}$ satisfying the following conditions.
\begin{enumerate}
\item For any $\gamma = [\begin{smallmatrix}a & b \\c & d \end{smallmatrix}] \in \mathrm{SL}_2(\mathbb{Z})$,
\begin{align*}
f\biggl(\frac{az+b}{cz+d}\biggr) = (cz+d)^k f(z).
\end{align*}
\item $f$ is annihilated by the weight $k$ hyperbolic Laplacian
\begin{align*}
\Delta_k := -y^2\biggl(\frac{\partial^2}{\partial x^2} + \frac{\partial^2}{\partial y^2}\biggr) +iky\biggl(\frac{\partial}{\partial x} + i\frac{\partial}{\partial y}\biggr).
\end{align*}
\item There exists an $\alpha \in \mathbb{R}$ such that $f(x+iy) = O(y^{\alpha})$ as $y \to \infty$, uniformly in $x \in \mathbb{R}$.
\end{enumerate}
The space\footnote{This definition of harmonic Maass forms adopted in Lagarias-Rhoades \cite{LR} is not the standard one. For example in Bringmann-Diamantis-Raum \cite{BDR}, a harmonic Maass form might have exponentially growing terms at the cusp.} of all such forms is denoted by $H_k$. Relaxing the condition (3) to $f(x+iy) = O(e^{\alpha y})$, we denote by $H_k^!$. In addition, we relax the condition (2) to $\Delta_k^r f = 0$ for $r \in \mathbb{Z}_{\geq 1}$, then we denote by $H_k^r$ and $H_k^{r,!}$, respectively, and we call a function $f \in H_k^r$ $(\text{resp. }H_k^{r,!})$ a polyharmonic (weak) Maass form of weight $k$ and depth $r$ (see \cite{ALR}, \cite{LR}). In particular we see that $H_k^1 = H_k$ and $H_k^{1,!} = H_k^!$. We next consider the real analytic Eisenstein series of weight $k \in 2\mathbb{Z}$ defined by
\begin{align*}
E_k(z,s) := \sum_{(m,n) \in \mathbb{Z}^2 \backslash (0,0)}\frac{y^s}{(mz+n)^k|mz+n|^{2s}},
\end{align*}
and its double completion 
\begin{align*}
\widehat{\widehat{E}}_k(z,s) := \bigl(s+\frac{k}{2}\bigr)\bigl(s+\frac{k}{2}-1\bigr)\pi^{-(s+\frac{k}{2})}\Gamma \bigl(s+\frac{k}{2}+\frac{|k|}{2}\bigr)E_k(z,s).
\end{align*}
(As a remark, this doubly completed Eisenstein series $\widehat{\widehat{E}}_0(z,s)$ coincides $F_z(s)$ in Brown-Omar \cite{BOm}). Then we also have the equation $\Delta_k \widehat{\widehat{E}}_k(z,s) = s(1-k-s)\widehat{\widehat{E}}_k(z,s)$. We define the Taylor coefficients of $\widehat{\widehat{E}}_{k}(z,s)$ by
\begin{align*}
\widehat{\widehat{E}}_{k}(z,s) = \left\{\begin{array}{ll}
\displaystyle{\sum_{r=0}^{\infty}} F_{k,r}(z)(s+k-1)^r \quad &\text{if } k \leq 0, \\
\displaystyle{\sum_{r=0}^{\infty}} G_{k,r}(z)s^r\quad &\text{if } k \geq 2,
\end{array} \right.
\end{align*}
then the following theorem was established.
\begin{thm}\cite{LR}\label{LRthm}
Let $r \geq 1$ be an integer, and $k \in 2\mathbb{Z}$ an even integer. Then
\begin{enumerate}
\item For an even integer $k \leq 0$, $\{F_{k,0}(z), \dots, F_{k,r-1}(z)\}$ is a basis for $H_k^r$.
\item For $k =2$, $\{G_{2,1}(z), \dots, G_{2,r}(z)\}$ is a basis for $H_2^r$.
\item For an even integer $k \geq 4$, it holds that $H_k^r = E_k^r + S_k$, where $E_k^r$ is spanned by $\{G_{k,0}(z), \dots G_{k,r-1}(z)\}$ and $S_k$ consists of cusp forms on $\mathrm{SL}_2(\mathbb{Z})$.
\end{enumerate}
\end{thm}
\begin{rmk*}
The first few coefficients are given by
\begin{align*}
F_{0,0}(z) &= 1,\\
F_{0,1}(z) &= \gamma +1 - \mathrm{log}(4\pi) - \mathrm{log}(y|\eta(z)|^4),\\
G_{2,0}(z) &= 0,\\
G_{2,1}(z) &= \frac{\pi}{3}-\frac{1}{y}-8\pi \biggl(\sum_{n=1}^{\infty}\sigma_1(n)q^n\biggr) =: \frac{\pi}{3}E_2^*(z),\\
G_{k,0}(z) &= \bigl(\frac{k}{2}-1\bigr)k!\pi^{-\frac{k}{2}}\zeta(k) E_k(z),\quad \text{for}\ k \geq 4,
\end{align*}
where $\sigma_1(n)$ is the sum of divisors of $n$, and $E_k(z)$ is the usual Eisenstein series of weight $k$, whose constant term is equal to $1$.
\end{rmk*}
Moreover we define polyharmonic Maass forms with half-integral depth. Let $\xi_k$ be the Maass type differential operator defined by
\begin{align*}
\xi_k := 2iy^k\overline{\frac{\partial}{\partial \bar{z}}}.
\end{align*}
This operator was introduced by Bruinier-Funke \cite{BF}. The important point is that this operator sends harmonic Maass forms of weight $k$ to holomorphic modular forms of weight $2-k$, and $\Delta_k = -\xi_{2-k} \circ \xi_k$ holds. Then for an integer $r \geq 1$, a polyharmonic Maass form $f(z)$ of depth $r - 1/2$ is characterized by (1) modularity, (3) moderate growth condition at the cusp, and (2)' $\xi_k \circ \Delta_k^{r-1} f(z) = 0$. Lagarias-Rhoades \cite{LR} obtained the recursion formulas among the above Taylor coefficients $F_{k,r}(z)$ and $G_{k,r}(z)$ as follows. 
\begin{align}\label{rec}
\begin{split}
\xi_k F_{k,r}(z) &= G_{2-k,r}(z),\\
\xi_k G_{k,r}(z) &= (k-1)F_{2-k,r-1}(z) + F_{2-k,r-2}(z),\\
\Delta_k F_{k,r}(z) &= (k-1)F_{k,r-1}(z) - F_{k,r-2}(z),\\
\Delta_k G_{k,r}(z) &= (1-k)G_{k,r-1}(z) - G_{k,r-2}(z).
\end{split}
\end{align}
Hence we can refine Theorem \ref{LRthm}.
\begin{thm}\cite{LR}\label{LRthm2}
Let $r \geq 1$ be an integer, and $k \in 2\mathbb{Z}$ an even integer. Then
\begin{enumerate}
\item For an even integer $k \leq -2$, $H_k^{1/2} =\{0\}$ and $\{F_{k,0}(z), \dots, F_{k,r-1}(z)\}$ is a basis for $H_k^r = H_k^{r+1/2}.$
\item For $k=0$, $\{F_{0,0}(z), \dots, F_{0,r-1}(z)\}$ is a basis for $H_0^r = H_0^{r-1/2}.$
\item For $k=2$, $H_2^{1/2} =\{0\}$ and $\{G_{2,1}(z), \dots, G_{2,r}(z)\}$ is a basis for $H_2^r = H_2^{r+1/2}$.
\item For an even integer $k \geq 4$, it holds that $H_k^r = H_k^{r-1/2} = E_k^r + S_k$, where $E_k^r$ is spanned by $\{G_{k,0}(z), \dots G_{k,r-1}(z)\}$ and $S_k$ consists of cusp forms on $\mathrm{SL}_2(\mathbb{Z})$.
\end{enumerate}
\end{thm}
By the above recursion formulas (\ref{rec}), for example we obtain the following ``Maass sequence",
\begin{align*}
\cdots \xrightarrow{\xi_{2-k}} \frac{1}{k-1} G_{k,1}(z) \xrightarrow{\xi_k} F_{2-k,0}(z) \xrightarrow{\xi_{2-k}} G_{k,0}(z) \xrightarrow{\xi_k} 0.
\end{align*}
In other words, we can construct preimages of the Eisenstein series under the $\xi$-operator. On the other hand as for the discriminant function $\Delta(z)$, Ono \cite{O} constructed its preimage $R_{\Delta}(z)$ satisfying $\xi_{-10}R_{\Delta}(z) = \Delta(z)$ up to a constant multiple. However it has an exponentially growing term, that is, $R_{\Delta}(z) \not\in H_{-10}^1$ but $R_{\Delta}(z) \in H_{-10}^{1,!}$. In order to construct a preimage of any polyharmonic Maass forms (of course, including holomorphic cusp forms) under the $\xi$-operator, it is necessary to generalize Theorem \ref{LRthm2} and the formulas (\ref{rec}) for the whole of polyharmonic weak Maass forms. \\

For the purpose of constructing polyharmonic weak Maass forms, we consider the Maass-Poincar\'{e} series. Let $\Gamma = \mathrm{SL}_2(\mathbb{Z})$, and $\Gamma_{\infty}$ the stabilizer of $\infty$ in $\Gamma$. For $k \in 2\mathbb{Z}$ and $m \in \mathbb{Z}_{\neq 0}$, we define the Maass-Poincar\'{e} series by
\begin{align*}
P_{k,m}(z,s) &:= \sum_{\gamma \in \Gamma_{\infty} \backslash \Gamma} \phi_{k,m}(z,s) |_k\gamma\\
&= \sum_{\gamma \in \Gamma_{\infty} \backslash \Gamma}  \Biggl((4\pi y)^{-\frac{k}{2}}M_{\mathrm{sgn}(m)\frac{k}{2},s-\frac{1}{2}}(4\pi |m|y)e^{2\pi imx}\Biggr)\Bigg|_k\gamma
\end{align*}
for $\mathrm{Re}(s) > 1$, where $|_k \gamma$ is the usual slash operator and $M_{\mu,\nu}(y)$ is the M-Whittaker function. This function satisfies $\Delta_k P_{k,m}(z,s) = (s-k/2)(1-k/2-s)P_{k,m}(z,s)$, and it has the following Taylor expansion form
\begin{align*}
P_{k,m}(z,s) = \left\{\begin{array}{ll}
\displaystyle{\sum_{r=0}^{\infty}} F_{k,m,r}(z)\bigl(s+\frac{k}{2}-1\bigr)^r \quad &\text{if } k \leq 0, \\
\displaystyle{\sum_{r=0}^{\infty}} G_{k,m,r}(z)\bigl(s-\frac{k}{2}\bigr)^r \quad &\text{if } k \geq 2.
\end{array} \right.
\end{align*}
If $k=0$ or $2$, this series needs the analytic continuation to $s=1$. Then these Taylor coefficients $F_{k,m,r}(z)$ and $G_{k,m,r}(z)$ are polyharmonic weak Maass forms of weight $k$ and depth $r+1$ in a similar manner, and satisfy the following recursion formulas.
\begin{thm}\label{Main2}
Let $k \in 2\mathbb{Z}$, $m \neq 0$, and $r \geq 0$ be integers. For the Taylor coefficients $F_{k,m,r}(z)$ and $G_{k,m,r}(z)$, they hold that
\begin{align*}
\xi_k F_{k,m,r}(z) &=(4\pi)^{1-k}\bigl\{(1-k)G_{2-k,-m,r}(z)+G_{2-k,-m,r-1}(z)\bigr\},\\
\xi_k G_{k,m,r}(z) &=(4\pi)^{1-k}F_{2-k,-m,r-1}(z),\\
\Delta_k F_{k,m,r}(z) &= (k-1)F_{k,m,r-1}(z) - F_{k,m,r-2}(z) ,\\
\Delta_k G_{k,m,r}(z) &= (1-k)G_{k,m,r-1}(z) - G_{k,m,r-2}(z).
\end{align*}
Here we put $F_{k,m,r}(z) = G_{k,m,r}(z) = 0$ for any $r<0$.
\end{thm}
By Theorem \ref{Main2}, for example we obtain the following ``Maass sequences".
\begin{align*}
\cdots \xrightarrow{\xi_{-10}} (4\pi)^{11} G_{12,1,1}(z) \xrightarrow{\xi_{12}} F_{-10,-1,0}(z) \xrightarrow{\xi_{-10}} 11(4\pi)^{11}G_{12,1,0}(z) \xrightarrow{\xi_{12}} 0.\\
\cdots \xrightarrow{\xi_0} 4\pi G_{2,1,2}(z) \xrightarrow{\xi_2} F_{0,-1,1}(z) \xrightarrow{\xi_0} 4\pi G_{2,1,1}(z) \xrightarrow{\xi_2} F_{0,-1,0}(z) \xrightarrow{\xi_0} 0.
\end{align*}
Here we use the fact that $G_{2,m,0}(z)$ with $m>0$ is equal to $0$. Furthermore we have the following theorem as an extension of Theorem \ref{LRthm2}.
\begin{thm}\label{Main1}
Let $r\geq 1$ be an integer. For an even integer $k \in 2\mathbb{Z}$, we define an integer $\ell_k$ by $k = 12\ell_k + k'$ where $k' \in \{0,4,6,8,10,14\}$. For each integer $m \geq -\ell_k$, the unique weakly holomorphic modular form of the form $f_{k,m}(z) = q^{-m} + \sum_{n > \ell_k}a_k(m,n)q^n$ is given. Then $\{ f_{k,m}(z)\ |\ m \geq -\ell_k\}$ is a basis of $H_k^{1/2,!}$. Moreover
\begin{enumerate}
\item For an even integer $k \leq 0$,
\begin{enumerate}
\item $\{ F_{k,m,r-1}(z)\ |\ m > \ell_k\}$ is a basis for $H_k^{r,!}/H_k^{r-1/2,!}$.
\item $\{ \tilde{F}_{k,m,r-1}(z)\ |\ m \leq \ell_k\}$ is a basis for $H_k^{r-1/2,!}/H_k^{r-1,!}$.
\end{enumerate}
\item For an even integer $k \geq 2$,
\begin{enumerate}
\item $\{ \tilde{G}_{k,m,r}(z)\ |\ m > \ell_k\}$ is a basis for $H_k^{r,!}/H_k^{r-1/2,!}$.
\item $\{ G_{k,m,r-1}(z)\ |\ m \leq \ell_k\}$ is a basis for $H_k^{r-1/2,!}/H_k^{r-1,!}$.
\end{enumerate}
\end{enumerate}
Here we put
\begin{align*}
\tilde{F}_{k,m,r-1}(z) &:= |m|^{-\frac{k}{2}}F_{k,m,r-1}(z) + \sum_{\ell_k<n<0}a_k(-m,n)|n|^{-\frac{k}{2}}F_{k,n,r-1}(z),\\
\tilde{G}_{k,m,r}(z) &:= m^{\frac{k}{2}-1}G_{k,m,r}(z) - \sum_{0<n \leq \ell_k}a_k(-n,m)n^{\frac{k}{2}-1}G_{k,n,r}(z).
\end{align*}
In addition, we put $F_{k,0,r}(z) := F_{k,r}(z)$, $G_{k,0,r}(z) := G_{k,r}(z)$, $\tilde{F}_{0,0,r-1}(z) := F_{0,r-1}(z)$, and $\tilde{G}_{2,0,r}(z) := G_{2,r}(z)$.
\end{thm}

\begin{rmk*}
It is known that $G_{12,1,0}(z)$ is equal to the discriminant function $\Delta(z)$ up to a constant multiple, and $F_{0,-1,0}(z) = j(z)-720$, the elliptic modular $j$-function whose constant term is equal to 24. The functions $f_{k,m}(z)$ in Theorem \ref{Main1} are called the Duke-Jenkins basis, (for more details, see Section \ref{s3.5}). 
\end{rmk*}

\begin{rmk*}
In the special case of $r=1$ and $k=2$, a basis for $H_2^{1,!}$ was constructed by Duke-Imamo\={g}lu-T\'{o}th \cite{DIT2}. More recently, for general $k \in \frac{1}{2}\mathbb{Z}$, Jeon-Kang-Kim \cite{JKK2} obtained a basis for $H_k^{1,!}$, and the author \cite{M} generalized their works to $H_k^{r,!}$ for any $k, r \in \frac{1}{2}\mathbb{Z}$.
\end{rmk*}

As a corollary of Theorem \ref{Main2} and Theorem \ref{Main1}, we obtain the following statement.
\begin{cor}
For an even integer $k \in 2\mathbb{Z}$ and any $r \in \frac{1}{2}\mathbb{Z}$, the map
\begin{align*}
\xi_k : H_k^{r,!} \twoheadrightarrow H_{2-k}^{r-1/2,!}
\end{align*}
is surjective.
\end{cor}

Historically speaking, Duke-Imamo\={g}lu-T\'{o}th \cite{DIT1} constructed an example of polyharmonic Maass form of depth $3/2$ whose Fourier coefficients encode real quadratic class numbers. After that Bringmann-Diamantis-Raum \cite{BDR} introduced general sesquiharmonic Maass forms (``sesqui" means depth $3/2$) to study the non-critical values of $L$-functions. (Poly-)Harmonic Maass forms have played important roles in the number theory. Bruinier-Ono \cite{BOn}, and more generally Alfes-Griffin-Ono-Rolen \cite{AGOR} established the connection between the central values $L(E_D, 1), L'(E_D,1)$ and the theory of harmonic Maass forms, where $L(E_D,s)$ is the Hasse-Weil $L$-function of the quadratic twist elliptic curve $E_D$. More precisely, they gave the canonical harmonic Maass forms which encode the vanishing or non-vanishing of $L(E_D, 1)$ and $L'(E_D,1)$. Other studies on polyharmonic Maass forms are given by Bruinier-Funke-Imamo\={g}lu \cite{BFI} and Ahlgren-Andersen-Samart \cite{AAS}.\\

Section \ref{s2} consists of basic properties of Whittaker functions and known results on polyharmonic Maass forms. In this section, we give the Fourier-Whittaker expansion of polyharmonic weak Maass forms. In Section \ref{s3}, we review the Maass-Poincar\'{e} series based on Duke-Imamo\={g}lu-T\'{o}th \cite{DIT2}. For the special case of depth $1/2$, that is, weakly holomorphic modular forms, Duke-Jenkins \cite{DJ} constructed a basis for the space of such forms. We recall their results in Section \ref{s3.5}, and reveal the connection between their basis $f_{k,m}(z)$ and our polyharmonic weak Maass forms $F_{k,m,r}(z)$ and $G_{k,m,r}(z)$. After that, in Section \ref{s4} and \ref{s5}, we give proofs of Theorem \ref{Main2} and \ref{Main1}, respectively. Finally in Section \ref{s6}, we show some examples of known polyharmonic weak Maass forms.

\begin{ack}
I would like to express my gratitude to Scott Ahlgren for letting me know polyharmonic Maass forms. Many thanks also to Daeyeol Jeon, Soon-Yi Kang, and Chang Heon Kim for their helpful comments on Theorem \ref{Main1}, and to Kathrin Bringmann and Jonas Kaszian for providing me a detailed history of polyharmonic Maass forms. Furthermore it is a pleasure to thank Shunsuke Yamana for inviting me to  MPIM, and also thank Soon-Yi Kang and Masanobu Kaneko for giving me the precious opportunity to discuss in NIMS. I am very grateful to the referee for careful reading of this paper and fruitful suggestions. This work is supported by Research Fellow (DC) of Japan Society for the Promotion of Science, and JSPS Overseas Challenge Program for Young Researchers.
\end{ack}

\section{Polyharmonic weak Maass forms and Whittaker functions}\label{s2}
In this section, we review the basic definitions and properties of Whittaker functions, and give the Fourier-Whittaker expansion of polyharmonic weak Maass forms.

\subsection{The Whittaker differential equation}\label{s21}
For given parameters $\mu, \nu \in \mathbb{C}$, the Whittaker differential equation \cite[(9.220)]{GR}, \cite[(7.1.1)]{MOS} is defined by
\begin{align}\label{WD}
D_{\mu, \nu}w(z) := \frac{d^2w}{dz^2} + \biggl(-\frac{1}{4} + \frac{\mu}{z} + \frac{1-4\nu^2}{4z^2} \biggr)w =0.
\end{align}
The standard solutions to this differential equation are given the Whittaker functions $M_{\mu,\nu}(z)$ and $W_{\mu,\nu}(z)$. For parameters $\mu, \nu$ with $\mathrm{Re}(\nu \pm \mu+1/2) >0$ and $y>0$, these Whittaker functions have integral representations \cite[(7.4.1), (7.4.2)]{MOS};
\begin{align*}
M_{\mu,\nu}(y) &= y^{\nu+\frac{1}{2}}e^{\frac{y}{2}}\frac{\Gamma(1+2\nu)}{\Gamma(\nu+\mu+\frac{1}{2})\Gamma(\nu-\mu+\frac{1}{2})}\int_0^1 t^{\nu+\mu-\frac{1}{2}}(1-t)^{\nu-\mu-\frac{1}{2}}e^{-yt}dt,\\
W_{\mu,\nu}(y) &= y^{\nu+\frac{1}{2}}e^{\frac{y}{2}}\frac{1}{\Gamma(\nu-\mu+\frac{1}{2})}\int_1^{\infty}t^{\nu+\mu-\frac{1}{2}}(t-1)^{\nu-\mu-\frac{1}{2}}e^{-yt}dt.
\end{align*}
Then from these integral representations, we can obtain their asymptotic behaviors as $y \to \infty$ (see \cite{DIT2}, \cite[(7.6.1)]{MOS});
\begin{align}\label{asym}
M_{\mu,\nu}(y) \sim \frac{\Gamma(1+2\nu)}{\Gamma(\nu-\mu+\frac{1}{2})}y^{-\mu}e^{\frac{y}{2}}\quad \text{and} \quad W_{\mu,\nu}(y) \sim y^{\mu}e^{-\frac{y}{2}}.
\end{align}
However since the Wronskian between $M_{\mu,\nu}(z)$ and $W_{\mu,\nu}(z)$ is given by \cite[(7.1.2)]{MOS}
\begin{align*}
\mathcal{W}(M_{\mu,\nu}(z), W_{\mu,\nu}(z)) = -\frac{\Gamma(2\nu+1)}{\Gamma(\nu-\mu+\frac{1}{2})},
\end{align*}
$M_{\mu,\nu}(z)$ and $W_{\mu,\nu}(z)$ are linearly dependent if the right-hand side vanishes. On the other hand, the function $W_{-\mu,\nu}(z e^{\pi i})$ is also a solution of (\ref{WD}), and the Wronskian is given by \cite[(7.1.2)]{MOS}
\begin{align*}
\mathcal{W}(W_{\mu,\nu}(z), W_{-\mu,\nu}(z e^{\pi i})) = e^{-\pi i \mu}.
\end{align*}
Then we can take $W_{\mu,\nu}(z)$ and $W_{-\mu,\nu}(z e^{\pi i})$ as linearly independent solutions to (\ref{WD}). According to the paper \cite{ALR}, we put $\mathcal{M}^+_{\mu,\nu}(z) := W_{-\mu,\nu}(z e^{\pi i})$. By \cite[(9.233)]{GR}, we see  
\begin{align}\label{MMW}
M_{\mu,\nu}(y) = \frac{\Gamma(1+2\nu)}{\Gamma(\nu-\mu+\frac{1}{2})}e^{\pi i\mu}\mathcal{M}^+_{\mu,\nu}(y) + \frac{\Gamma(1+2\nu)}{\Gamma(\nu+\mu+\frac{1}{2})}e^{-\pi i (\nu-\mu+\frac{1}{2})}W_{\mu,\nu}(y),
\end{align}
for $2\nu \not\in \mathbb{Z}_{<0}$ and $y>0$.

\subsection{The Fourier-Whittaker expansion}\label{s22}
As we mentioned in Section \ref{s1}, a polyharmonic weak Maass form $f$ of weight $k \in 2\mathbb{Z}$ and depth $r \in \mathbb{Z}_{\geq 1}$ is defined by the following conditions,
\begin{enumerate}
\item For any $\gamma = [\begin{smallmatrix}a & b \\c & d \end{smallmatrix}] \in \mathrm{SL}_2(\mathbb{Z})$,
\begin{align*}
f\biggl(\frac{az+b}{cz+d}\biggr) = (cz+d)^k f(z).
\end{align*}
\item $f(z)$ is smooth, and satisfies
\begin{align*}
\Delta_k^r f(z) = 0.
\end{align*}
\item There exists an $\alpha \in \mathbb{R}$ such that $f(x+iy) = O(e^{\alpha y})$ as $y \to \infty$, uniformly in $x \in \mathbb{R}$.
\end{enumerate}

We now explain the Fourier-Whittaker expansion of $f(z) \in H_k^{r,!}$.  The key point is that the operator $\frac{\partial}{\partial s}$ commutes with $\Delta_k$. In order to obtain our goal, we recall the functions $u_{k,n}^{[j],\pm}(y)$ defined by
\begin{align*}
u_{k,n}^{[j],-}(y) &:= y^{-\frac{k}{2}}\frac{\partial^j}{\partial s^j}W_{\mathrm{sgn}(n)\frac{k}{2}, s-\frac{1}{2}}(4\pi |n|y)\bigg|_{s=\frac{k}{2}},\\
u_{k,n}^{[j],+}(y) &:= y^{-\frac{k}{2}}\frac{\partial^j}{\partial s^j}\mathcal{M}^+_{\mathrm{sgn}(n)\frac{k}{2}, s-\frac{1}{2}}(4\pi |n|y)\bigg|_{s=\frac{k}{2}}
\end{align*}
for $n \in \mathbb{Z}_{\neq 0}$, and for $n=0$ we put 
\begin{align*}
u_{k,0}^{[j],-}(y) &:= \frac{\partial^j}{\partial s^j}y^{1-\frac{k}{2}-s}\bigg|_{s=\frac{k}{2}} = (-1)^j(\mathrm{log}\ y)^j y^{1-k},\\
u_{k,0}^{[j],+}(y) &:= \frac{\partial^j}{\partial s^j}y^{s-\frac{k}{2}}\bigg|_{s=\frac{k}{2}} = (\mathrm{log}\ y)^j.
\end{align*}
These functions are the special cases $u_{k,n}^{[j],\pm}(y;0)$ of the functions introduced in \cite[(3.4), (3.5)]{ALR}. As a remark, it is known that \cite[(2.11), (2.12)]{DIT2}, \cite[(7.2.4)]{MOS}
\begin{align*}
u_{k,n}^{[0],-}(y)e^{2\pi inx} &= y^{-\frac{k}{2}}W_{\mathrm{sgn}(n)\frac{k}{2}, \frac{k-1}{2}}(4\pi |n|y)e^{2\pi inx}\\
&= \left\{\begin{array}{ll}
(4\pi n)^{\frac{k}{2}}q^n\quad &\text{if } n > 0, \\
(4\pi |n|)^{\frac{k}{2}}\Gamma(1-k,4\pi |n|y)q^n\quad &\text{if } n<0,
\end{array} \right.\\
u_{k,n}^{[0],+}(y)e^{2\pi inx} &= y^{-\frac{k}{2}}\mathcal{M}^+_{-\frac{k}{2}, \frac{k-1}{2}}(4\pi |n|y)e^{2\pi inx}\\
&= (4\pi n)^{\frac{k}{2}} q^n,\hspace{105pt} \text{if } n<0,
\end{align*}
where $\Gamma(s,y) := \int_y^{\infty} e^{-t}t^{s-1}dt$ is the incomplete Gamma function. By \cite[Corollary A.2]{ALR}, the set
\begin{align*}
\biggl\{u_{k,n}^{[j],-}(y)\ \bigg|\ 0 \leq j \leq r\biggr\} \cup \biggl\{u_{k,n}^{[j],+}(y)\ \bigg|\ 0 \leq j \leq r\biggr\}
\end{align*}
is linearly independent for each integer $r$. Then by the conditions (1) and (2), a polyharmonic weak Maass form has the following Fourier-Whittaker expansion form.
\begin{prp}\label{FWE}\cite[Section 3]{ALR}
Let $f(z) \in H_k^{r,!}$ for $k \in 2\mathbb{Z}$ and $r \in \mathbb{Z}_{\geq 1}$. Then the Fourier-Whittaker expansion of $f(z)$ is given by
\begin{align*}
f(z) = \sum_{n = -\infty}^{\infty}\sum_{j=0}^{r-1} \biggl(c_{n,j}^- u_{k,n}^{[j],-}(y)e^{2\pi i n x} + c_{n,j}^+ u_{k,n}^{[j],+}(y)e^{2\pi i n x} \biggr),
\end{align*}
where $c_{n,j}^{\pm} \in \mathbb{C}$.
\end{prp}
Furthermore by Corollary A.3 in \cite{ALR}, for $n \neq 0$, $u_{k,n}^{[j],-}(y)$ decays exponentially as $y \to \infty$, while $u_{k,n}^{[j],+}(y)$ grows exponentially as $y \to \infty$. Combining with the condition (3), we see that the Fourier-Whittaker coefficients $c_{n,j}^+ = 0$ for almost all indices $(n,j)$. If all coefficients $c_{n,j}^+ =0$ for $n \neq 0$, then $f \in H_k^r$.

\subsection{The action of $\xi$-operator}
First, we recall some basic properties of the Maass type differential operator $\xi_k$. For a $C^{\infty}$-function $f: \mathfrak{H} \to \mathbb{C}$, any $\gamma = [\begin{smallmatrix}a & b \\c & d \end{smallmatrix}] \in \mathrm{SL}_2(\mathbb{Z})$ and integer $k$, it can be easily checked that
\begin{align}\label{xi}
\xi_k\biggl((cz+d)^{-k}f(\gamma z)\biggr) = (cz+d)^{k-2}(\xi_k f)(\gamma z).
\end{align}
Thus we obtain the most important property, if $f(z)$ is a modular form of weight $k$, then $\xi_kf$ is of weight $2-k$. Moreover, $\xi_kf = 0$ if and only if $f$ is holomorphic. Since a harmonic Maass form is characterized by $\Delta_k f = -\xi_{2-k}\circ \xi_k f = 0$, we see that $\xi_k$ maps $H_k^{1,!}$ to $H_{2-k}^{1/2,!} = M_{2-k}^!$ with kernel $M_k^!$. Here we denote by $M_k^!$ the space of weakly holomorphic modular forms of weight $k$.\\

For general $H_k^{r,!}$ terms, we show the following Lemma.
\begin{lmm}\label{xiop}
Under the above notations, we have
\begin{align*}
&\quad\xi_k(u_{k,n}^{[j],-}(y)e^{2\pi i nx})\\
&\quad\quad=\left\{\begin{array}{ll}
j(1-k)u_{2-k,-n}^{[j-1],-}(y)e^{-2\pi i nx} - j(j-1)u_{2-k,-n}^{[j-2],-}(y)e^{-2\pi inx} \quad \text{if } n > 0, \\
-u_{2-k,-n}^{[j],-}(y)e^{-2\pi inx} \quad \text{if } n < 0,
\end{array} \right.\\
&\quad\xi_k(u_{k,n}^{[j],+}(y)e^{2\pi i nx})\\
&\quad\quad = \left\{\begin{array}{ll}
-u_{2-k,-n}^{[j],+}(y)e^{-2\pi inx} \quad \text{if } n > 0,\\
j(1-k)u_{2-k,-n}^{[j-1],+}(y)e^{-2\pi i nx} - j(j-1)u_{2-k,-n}^{[j-2],+}(y)e^{-2\pi inx} \quad \text{if } n < 0,
\end{array} \right.\\
&\quad\xi_k(u_{k,0}^{[j],-}(y)) = (-1)^j \biggl( j u_{2-k,0}^{[j-1],+}(y) + (1-k) u_{2-k,0}^{[j],+}(y) \biggr),\\
&\quad\xi_k(u_{k,0}^{[j],+}(y)) = (-1)^{j-1} j u_{2-k,0}^{[j-1],-}(y),
\end{align*}
where we put $u_{k,n}^{[j],\pm}(y) = 0$ for any $j<0$.
\end{lmm}

\begin{proof}
By the commutativity $\xi_k \frac{\partial}{\partial s} = \frac{\partial}{\partial \bar{s}} \xi_k$, we have
\begin{align*}
\xi_k(u_{k,n}^{[j],-}(y)e^{2\pi i nx}) &= \xi_k \Biggl(y^{-\frac{k}{2}}\frac{\partial^j}{\partial s^j}W_{\mathrm{sgn}(n)\frac{k}{2}, s-\frac{1}{2}}(4\pi |n|y)\bigg|_{s=\frac{k}{2}} e^{2\pi inx} \Biggr)\\
&= \frac{\partial^j}{\partial \bar{s}^j} \xi_k \Biggl(y^{-\frac{k}{2}}e^{2\pi ny}W_{\mathrm{sgn}(n)\frac{k}{2}, s-\frac{1}{2}}(4\pi |n|y) e^{2\pi inx - 2\pi ny} \Biggr)\bigg|_{\bar{s}=\frac{k}{2}}.
\end{align*}
For $n >0$, 
\begin{align*}
&\xi_k \Biggl(y^{-\frac{k}{2}}e^{2\pi ny}W_{\frac{k}{2}, s-\frac{1}{2}}(4\pi ny) e^{2\pi inx-2\pi ny} \Biggr)\\
&\quad= iy^k\overline{\biggl(\frac{\partial}{\partial x} + i \frac{\partial}{\partial y}\biggr)}\Biggl(y^{-\frac{k}{2}}e^{2\pi ny}W_{\frac{k}{2}, s-\frac{1}{2}}(4\pi ny) e^{2\pi i nz} \Biggr)\\
&\quad= (4\pi n)^{\frac{k}{2}}y^k \overline{\frac{\partial}{\partial y}\biggl((4\pi ny)^{-\frac{k}{2}}e^{2\pi ny}W_{\frac{k}{2},s-\frac{1}{2}}(4\pi ny)\biggr)}\cdot e^{-2\pi in\bar{z}}.
\end{align*}
By the formula \cite[(7.2.1)]{MOS},
\begin{align*}
\frac{d}{dz} \biggl(z^{-\mu}e^{\frac{1}{2}z}W_{\mu,\nu}(z)\biggr) = \bigl(\frac{1}{2}+\nu-\mu\bigr)\bigl(\frac{1}{2}-\nu-\mu\bigr)z^{-\mu-1}e^{\frac{1}{2}z}W_{\mu-1,\nu}(z),
\end{align*}
it holds that
\begin{align*}
&\frac{\partial}{\partial y}\biggl((4\pi ny)^{-\frac{k}{2}}e^{2\pi ny}W_{\frac{k}{2},s-\frac{1}{2}}(4\pi ny)\biggr)\\ 
&\quad= \frac{1}{y}\bigl(s-\frac{k}{2}\bigr)\bigl(1-\frac{k}{2}-s\bigr) e^{2\pi ny}(4\pi ny)^{-\frac{k}{2}}W_{\frac{k}{2}-1,s-\frac{1}{2}}(4\pi ny),
\end{align*}
that is, we obtain
\begin{align*}
\xi_k \Biggl(y^{-\frac{k}{2}}W_{\frac{k}{2}, s-\frac{1}{2}}(4\pi ny) e^{2\pi inx} \Biggr) &= y^{\frac{k}{2}-1}\bigl(\bar{s}-\frac{k}{2}\bigr)\bigl(1-\frac{k}{2}-\bar{s}\bigr) W_{\frac{k}{2}-1,\bar{s}-\frac{1}{2}}(4\pi ny) e^{-2\pi inx}.
\end{align*}
For the case of $n <0$,
\begin{align*}
\xi_k \Biggl(y^{-\frac{k}{2}}W_{-\frac{k}{2}, s-\frac{1}{2}}(-4\pi ny) e^{2\pi inx} \Biggr) &= y^k \overline{\frac{\partial}{\partial y}\biggl(y^{-\frac{k}{2}}e^{2\pi ny}W_{-\frac{k}{2},s-\frac{1}{2}}(-4\pi ny)\biggr)} \cdot e^{-2\pi in\bar{z}}.
\end{align*}
By using \cite[(7.2.1)]{MOS},
\begin{align*}
\frac{d}{dz} \biggl(z^{\mu}e^{-\frac{1}{2}z}W_{\mu,\nu}(z)\biggr) = -z^{\mu-1}e^{-\frac{1}{2}z}W_{\mu+1,\nu}(z),
\end{align*}
similarly we have
\begin{align*}
&\xi_k(u_{k,n}^{[j],-}(y)e^{2\pi i nx})\\ 
&\quad= \left\{\begin{array}{ll}
\dfrac{\partial^j}{\partial \bar{s}^j}\biggl(\bigl(\bar{s}-\dfrac{k}{2}\bigr)\bigl(1-\dfrac{k}{2}-\bar{s}\bigr)y^{-\frac{2-k}{2}}W_{-\frac{2-k}{2},\bar{s}-\frac{1}{2}}(4\pi ny)e^{-2\pi inx}\biggr) \bigg|_{\bar{s}=\frac{k}{2}}\quad &\text{if } n > 0, \\
\dfrac{\partial^j}{\partial \bar{s}^j}\biggl(-y^{-\frac{2-k}{2}}W_{\frac{2-k}{2},\bar{s}-\frac{1}{2}}(4\pi |n|y)e^{-2\pi inx}\biggr) \bigg|_{\bar{s}=\frac{k}{2}}\quad &\text{if } n < 0.
\end{array} \right.
\end{align*}
If $n > 0$, then
\begin{align*}
\xi_k(u_{k,n}^{[j],-}(y)e^{2\pi i nx}) &= \Biggl(j \bigl(1-\frac{k}{2}-\bar{s}\bigr)\frac{\partial^{j-1}}{\partial \bar{s}^{j-1}}\biggl(y^{-\frac{2-k}{2}}W_{-\frac{2-k}{2},\bar{s}-\frac{1}{2}}(4\pi ny)e^{-2\pi inx}\biggr)\\
&\hspace{40pt}- j(j-1)\frac{\partial^{j-2}}{\partial \bar{s}^{j-2}}\biggl(y^{-\frac{2-k}{2}}W_{-\frac{2-k}{2},\bar{s}-\frac{1}{2}}(4\pi ny)e^{-2\pi inx}\biggr) \Biggr)\bigg|_{\bar{s}=\frac{k}{2}}\\
&= j(1-k)u_{2-k,-n}^{[j-1],-}(y)e^{-2\pi i nx} - j(j-1)u_{2-k,-n}^{[j-2],-}(y)e^{-2\pi inx}.
\end{align*}
If $n < 0$, we have
\begin{align*}
\xi_k(u_{k,n}^{[j],-}(y)e^{2\pi i nx}) &= -u_{2-k,-n}^{[j],-}(y)e^{-2\pi inx}.
\end{align*}
For the case of $u_{k,n}^{[j],+}(y)$, we can calculate them by using
\begin{align*}
\frac{d}{dz}\biggl(z^{-\mu}e^{\frac{1}{2}z}\mathcal{M}_{\mu,\nu}^+(z)\biggr) &= -z^{-\mu-1}e^{\frac{1}{2}z}\mathcal{M}_{\mu-1,\nu}^+(z),\\
\frac{d}{dz}\biggl(z^{\mu}e^{-\frac{1}{2}z}\mathcal{M}_{\mu,\nu}^+(z)\biggr) &= \bigl(\frac{1}{2}+\nu+\mu\bigr)\bigl(\frac{1}{2}-\nu+\mu \bigr)z^{\mu-1}e^{-\frac{1}{2}z}\mathcal{M}_{\mu+1,\nu}^+(z).
\end{align*}
For the remaining cases, we can get them by direct calculations.
\end{proof} 

Combining Proposition \ref{FWE} and Lemma \ref{xiop}, we obtain
\begin{prp}\label{xilift}
The $\xi_k$ operator sends polyharmonic weak Maass forms to polyharmonic weak Maass forms, that is,
\begin{align*}
\xi_k(H_k^{r,!}) \subset H_{2-k}^{r-1/2,!}\quad \text{and}\quad \Delta_k(H_k^{r,!}) \subset H_k^{r-1,!}.
\end{align*}
In particular for a function $f \in H_k^{r,!}$ with an integer $r$, if $c_{n,r-1}^- = 0$ for all $n \leq 0$ and $c_{n,r-1}^+ =0$ for all $n>0$ in Proposition $\ref{FWE}$, then it strictly holds
\begin{align}\label{halfdep}
f \in H_k^{r-1/2,!}
\end{align}
and the space of polyharmonic weak Maass forms of depth $r-1/2$ consists of such forms.
\end{prp}

\begin{rmk*}
For the proof of this proposition, the $\xi$-operator is applied to the infinite Fourier-Whittaker expansion of Proposition \ref{FWE}. This is guaranteed as follows. We consider the Fourier expansion of the form $f(z) = \sum_{n \in \mathbb{Z}}a(n,y)e^{2\pi inx}$. Since a polyharmonic weak Maass form $f(z)$ is smooth, by the general theory of Fourier expansion, its termwise derivatives in $x$ is valid. As for its termwise derivatives in $y$, by the definition of the Fourier coefficients, we have
\begin{align*}
\frac{\partial}{\partial y}a(n,y) = \int_0^1 \frac{\partial f(x+iy)}{\partial y}e^{-2\pi inx}dx,
\end{align*}
that is, $\frac{\partial}{\partial y}a(n,y)$ is the $n$-th Fourier coefficient of $\frac{\partial}{\partial y}f(x+iy)$.
\end{rmk*}

\section{Maass-Poincar\'{e} series}\label{s3}
In this section, we consider the Maass-Poincar\'{e} series as a generalization of Niebur's Poincar\'{e} series \cite{N}, and compute its Taylor expansion with respect to $s$. Let $\Gamma = \mathrm{SL}_2(\mathbb{Z})$. For an even integer $k$ and nonzero integer $m$, let
\begin{align*}
\phi_{k,m}(z,s) := (4\pi y)^{-\frac{k}{2}} M_{\mathrm{sgn}(m)\frac{k}{2},s-\frac{1}{2}}(4\pi |m|y)e^{2\pi imx}
\end{align*}
and define the corresponding Poincar\'{e} series
\begin{align*}
P_{k,m}(z,s) := \sum_{\gamma \in \Gamma_{\infty} \backslash \Gamma} \phi_{k,m}(z,s) |_k\gamma.
\end{align*}
This series converges for $\mathrm{Re}(s) > 1$. We call $P_{k,m}(z,s)$ the Maass-Poincar\'{e} series of weight $k$ and index $m$.

\begin{lmm}\label{xiphi}
For a nonzero integer $m$, we have
\begin{align*}
\xi_k \phi_{k,m}(z,s) = (4\pi)^{1-k}\bigl(\bar{s}-\frac{k}{2}\bigr)\phi_{2-k,-m}(z,\bar{s}).
\end{align*}
\end{lmm}

\begin{proof}
By \cite[(7.2.1)]{MOS}, they hold that
\begin{align*}
\frac{d}{dz} \biggl(z^{-\mu}e^{\frac{z}{2}}M_{\mu,\nu}(z)\biggr) &= \bigl(\frac{1}{2} + \nu-\mu\bigr)z^{-\mu-1}e^{\frac{z}{2}}M_{\mu-1,\nu}(z),\\
\frac{d}{dz} \biggl(z^{\mu}e^{-\frac{z}{2}}M_{\mu,\nu}(z)\biggr) &= \bigl(\frac{1}{2} + \nu+\mu\bigr)z^{\mu-1}e^{-\frac{z}{2}}M_{\mu+1,\nu}(z).
\end{align*}
Applying these formulas to $\xi_k\phi_{k,m}(z,s)$, then we have
\begin{align*}
\xi_k\phi_{k,m}(z,s) &= iy^k\overline{\biggl(\frac{\partial}{\partial x} + i \frac{\partial}{\partial y}\biggr)}|m|^{\frac{k}{2}}(4\pi |m|y)^{-\frac{k}{2}}e^{2\pi my}M_{\mathrm{sgn}(m)\frac{k}{2}, s-\frac{1}{2}}(4\pi |m| y)e^{2\pi i mz}\\
&= y^k  |m|^{\frac{k}{2}} \overline{\frac{\partial}{\partial y}\biggl((4\pi|m| y)^{-\frac{k}{2}}e^{2\pi my}M_{\mathrm{sgn}(m)\frac{k}{2}, s-\frac{1}{2}}(4\pi |m| y)\biggr)} \cdot e^{-2\pi i m\bar{z}}\\
&= (4\pi)^{1-k} \bigl(\bar{s}-\frac{k}{2}\bigr) (4\pi y)^{\frac{k}{2}-1}M_{\mathrm{sgn}(-m)\frac{2-k}{2}, \bar{s}-\frac{1}{2}}(4\pi |m| y)e^{-2\pi i m x}\\
&= (4\pi)^{1-k}\bigl(\bar{s}-\frac{k}{2}\bigr)\phi_{2-k,-m}(z,\bar{s}).
\end{align*}
\end{proof}

From (\ref{xi}) and Lemma \ref{xiphi}, for a nonzero integer $m$, we see that
\begin{align}\label{xiP}
\xi_k P_{k,m}(z,s) = (4\pi)^{1-k}\bigl(\overline{s}-\frac{k}{2}\bigr)P_{2-k,-m}(z,\overline{s}),
\end{align}
and
\begin{align}\label{DelP}
\Delta_k P_{k,m}(z,s) = -\xi_{2-k} \circ \xi_k P_{k,m}(z,s) = \bigl(s-\frac{k}{2}\bigr)\bigl(1-\frac{k}{2}-s\bigr) P_{k,m}(z,s).
\end{align}
Thus $P_{k,m}(z,s)$ is harmonic at $s= k/2, 1-k/2$. Substituting $s=k/2$ if $k>2$, or $s = 1-k/2$ if $k < 0$, we immediately obtain harmonic forms. In the cases of $k=0$ and $2$, Duke-Imamo\={g}lu-T\'{o}th \cite{DIT2} gave the analytic continuation for $P_{k,m}(z,s)$ to $s=1$. Thus we can include the cases of $k=0, 2$.

\begin{rmk*}
For the real analytic Eisenstein series $E_k(z,s)$, they hold that
\begin{align*}
\xi_kE_k(z,s) &= \bar{s}E_{2-k}(z,\bar{s}+k-1),\\
\Delta_kE_k(z,s) &= s(1-k-s)E_k(z,s).
\end{align*}
\end{rmk*}

By using Kohnen's approach \cite{Kohnen}, we can compute the Fourier expansion of $P_{k,m}(z,s)$ expressed in terms of the Kloosterman sums and Bessel functions. Further details are explained in \cite[Section 3]{JKK3}.
\begin{prp}\label{Fourier} $($This form is found in \cite[Proposition 2]{DIT2}$)$.
For an even integer $k$ and a nonzero integer $m$, the Maass-Poincar\'{e} series $P_{k,m}(z,s)$ has the following Fourier expansion form.
\begin{align*}
P_{k,m}(z,s) = &(4\pi y)^{-\frac{k}{2}}M_{\mathrm{sgn}(m)\frac{k}{2}, s-\frac{1}{2}}(4\pi |m| y)e^{2\pi i m x} + g_{k,m,0}^{}(s)L_{m,0}(s)(-4\pi y)^{-\frac{k}{2}}y^{1-s}\\
&+ \sum_{n \neq 0}g_{k,m,n}^{}(s)L_{m,n}(s)(-4\pi y)^{-\frac{k}{2}}W_{\mathrm{sgn}(n)\frac{k}{2}, s-\frac{1}{2}}(4\pi|n|y)e^{2\pi i n x}
\end{align*}
where
\begin{align*}
g_{k,m,n}^{}(s) := \Gamma(2s)\times \left\{\begin{array}{ll}
\dfrac{2\pi \sqrt{|m/n|}}{\Gamma(s+\mathrm{sgn}(n)k/2)}\quad &\text{if } n \neq 0, \\
\dfrac{4\pi^{1+s}|m|^{s}}{(2s-1)\Gamma(s+k/2)\Gamma(s-k/2)}\quad &\text{if } n=0,
\end{array} \right.
\end{align*}
and
\begin{align*}
L_{m,n}(s) := \left\{\begin{array}{lll}
\displaystyle{\sum_{c=1}^{\infty}}\dfrac{K(m,n,c)}{c}J_{2s-1}\bigl(\dfrac{4\pi \sqrt{|mn|}}{c}\bigr)\quad &\text{if } nm > 0, \\
\displaystyle{\sum_{c=1}^{\infty}}\dfrac{K(m,0,c)}{c^{2s}}\quad &\text{if } n=0,\\
\displaystyle{\sum_{c=1}^{\infty}}\dfrac{K(m,n,c)}{c}I_{2s-1}\bigl(\dfrac{4\pi \sqrt{|mn|}}{c}\bigr)\quad &\text{if } nm < 0,
\end{array} \right.
\end{align*}
where $I_s(x)$ and $J_s(x)$ are two types of Bessel functions, and $K(m,n,c)$ is the Kloosterman sum
\begin{align*}
K(m,n,c) := \sum_{\substack{d (c)^*\\ad \equiv 1 (c)}} e^{2\pi i \frac{ma + nd}{c}}.
\end{align*}
\end{prp}

\section{Duke-Jenkins basis}\label{s3.5}
Duke-Jenkins \cite{DJ} constructed a basis for the space $H_k^{1/2,!} = M_k^!$ of weakly holomorphic modular forms. For an even integer $k \in 2\mathbb{Z}$, we define an integer $\ell_k$ by $k=12\ell_k + k'$ where $k' \in \{0, 4, 6, 8, 10, 14\}$. For each integer $m \geq -\ell_k$, there exists the unique weakly holomorphic modular form $f_{k,m}(z)$ with the Fourier expansion of the form
\begin{align*}
f_{k,m}(z) = q^{-m} + \sum_{n > \ell_k} a_k(m,n)q^n.
\end{align*}
Then they showed the set $\{ f_{k,m}(z)\ |\ m \geq -\ell_k\}$ is a basis of $M_k^!$. In the case of $k=0$, the  first few of the basis are given by
\begin{align*}
f_{0,0}(z) &= 1,\\
f_{0,1}(z) &= j(z) - 744 = q^{-1} + 196884q + 21493760q^2 + \cdots,\\
f_{0,2}(z) &= j(z)^2-1488j(z) + 159768 = q^{-2} + 42987520q + 40491909396q^2 + \cdots.
\end{align*}
For a general weight $k$, we put $f_k(z) = \Delta(z)^{\ell_k} E_{k'}(z)$, where $E_0(z) := 1$. Then the functions $f_{k,m}(z)$ are characterized by the following generating function
\begin{align*}
\sum_{m \geq -\ell_k} f_{k,m}(\tau)q^m = \frac{f_k(\tau)f_{2-k}(z)}{j(z)-j(\tau)},
\end{align*}
(see \cite[Theorem 2]{DJ}). Hence it follows that
\begin{align}\label{dual}
a_k(m,n) = -a_{2-k}(n,m).
\end{align}
We describe these basis functions $f_{k,m}(z)$ in terms of our functions $F_{k,m,r}(z)$ and $G_{k,m,r}(z)$. Combining the Fourier expansion form in Proposition \ref{Fourier} and (\ref{MMW}), we have 
\begin{align*}
P_{k,m}(z,s) &= (-4\pi y)^{-\frac{k}{2}}\Biggl(\frac{\Gamma(2s)}{\Gamma(s+\mathrm{sgn}(m)k/2)}e^{-s\pi i}W_{\mathrm{sgn}(m)\frac{k}{2},s-\frac{1}{2}}(4\pi |m|y) \\
&\hspace{100pt}+ \frac{\Gamma(2s)}{\Gamma(s-\mathrm{sgn}(m)k/2)}\mathcal{M}^+_{\mathrm{sgn}(m)\frac{k}{2},s-\frac{1}{2}}(4\pi |m|y)\Biggr)e^{2\pi i m x} \\
&\quad+ g_{k,m,0}^{}(s)L_{m,0}(s)(-4\pi y)^{-\frac{k}{2}}y^{1-s}\\
&\quad+ \sum_{n \neq 0}g_{k,m,n}^{}(s)L_{m,n}(s)(-4\pi y)^{-\frac{k}{2}}W_{\mathrm{sgn}(n)\frac{k}{2}, s-\frac{1}{2}}(4\pi|n|y)e^{2\pi i n x}.
\end{align*}
For a nonzero integer $m$, the Taylor coefficients $F_{k,m,r}(z)$ and $G_{k,m,r}(z)$ are given by 
\begin{align*}
F_{k,m,r}(z) &= \frac{1}{r!}\frac{\partial^r}{\partial s^r} P_{k,m}(z,s)|_{s=1-\frac{k}{2}}\quad \text{for}\ k \leq 0,\\
G_{k,m,r}(z) &= \frac{1}{r!}\frac{\partial^r}{\partial s^r} P_{k,m}(z,s)|_{s=\frac{k}{2}}\quad \text{for}\ k \geq 2.
\end{align*}
As remarks, first for $k=0, 2$, this Fourier expansion form gives the analytic continuation of $P_{k,m}(z,s)$ to $s=1$, that is, its termwise derivatives in $s$ is valid, (see \cite[Section 3.2]{DIT2}). Secondary, for a nonzero integer $m$, this Poincar\'{e} series $P_{k,m}(z,s)$ has no pole at $s=k/2$ for $k\geq 2$, and $s=1-k/2$ for $k \leq 0$. Then we have the Fourier-Whittaker expansion forms of $F_{k,m,r}(z)$ and $G_{k,m,r}(z)$,
\begin{align}\label{FFourier}
\begin{split}
r! (-4\pi)^{\frac{k}{2}}F_{k,m,r}(z) 
&=\sum_{j=0}^r (-1)^j {{r}\choose{j}}\frac{\partial^{r-j}}{\partial s^{r-j}} \biggl(\frac{\Gamma(2s)}{\Gamma(s+\mathrm{sgn}(m)k/2)}e^{-s\pi i}\biggr)\bigg|_{s=1-\frac{k}{2}} u_{k,m}^{[j],-}(y)e^{2\pi i m x}\\
&+ \sum_{j=0}^r (-1)^j {{r}\choose{j}}\frac{\partial^{r-j}}{\partial s^{r-j}} \biggl( \frac{\Gamma(2s)}{\Gamma(s-\mathrm{sgn}(m)k/2)}\biggr)\bigg|_{s=1-\frac{k}{2}} u_{k,m}^{[j],+}(y)e^{2\pi i m x} \\
&+ \sum_{j=0}^r (-1)^j {{r}\choose{j}}\frac{\partial^{r-j}}{\partial s^{r-j}}\biggl(g_{k,m,0}^{}(s)L_{m,0}(s)\biggr)\bigg|_{s=1-\frac{k}{2}} u_{k,0}^{[j],+}(y)\\
&+ \sum_{j=0}^r\sum_{n \neq 0} (-1)^j{{r}\choose{j}}\frac{\partial^{r-j}}{\partial s^{r-j}} \biggl(g_{k,m,n}^{}(s)L_{m,n}(s) \biggr)\bigg|_{s=1-\frac{k}{2}}u_{k,n}^{[j],-}(y)e^{2\pi inx},
\end{split}
\end{align}
and
\begin{align}\label{GFourier}
\begin{split}
r! (-4\pi)^{\frac{k}{2}}G_{k,m,r}(z) 
&=\sum_{j=0}^r {{r}\choose{j}}\frac{\partial^{r-j}}{\partial s^{r-j}} \biggl(\frac{\Gamma(2s)}{\Gamma(s+\mathrm{sgn}(m)k/2)}e^{-s\pi i}\biggr)\bigg|_{s=\frac{k}{2}} u_{k,m}^{[j],-}(y)e^{2\pi i m x}\\
&+ \sum_{j=0}^r {{r}\choose{j}}\frac{\partial^{r-j}}{\partial s^{r-j}} \biggl( \frac{\Gamma(2s)}{\Gamma(s-\mathrm{sgn}(m)k/2)}\biggr)\bigg|_{s=\frac{k}{2}} u_{k,m}^{[j],+}(y)e^{2\pi i m x} \\
&+ \sum_{j=0}^r {{r}\choose{j}}\frac{\partial^{r-j}}{\partial s^{r-j}}\biggl(g_{k,m,0}^{}(s)L_{m,0}(s)\biggr)\bigg|_{s=\frac{k}{2}} u_{k,0}^{[j],-}(y)\\
&+ \sum_{j=0}^r\sum_{n \neq 0}{{r}\choose{j}} \frac{\partial^{r-j}}{\partial s^{r-j}} \biggl(g_{k,m,n}^{}(s)L_{m,n}(s) \biggr)\bigg|_{s=\frac{k}{2}}u_{k,n}^{[j],-}(y)e^{2\pi inx},
\end{split}
\end{align}
where we use the property $\frac{\partial^j}{\partial \nu^j}W_{\mu,\nu}(z)|_{\nu = \nu_0} = (-1)^j\frac{\partial^j}{\partial \nu^j}W_{\mu,\nu}(z)|_{\nu=-\nu_0}$. By these Fourier-Whittaker expansion forms, we see that $F_{k,m,r}(z), G_{k,m,r}(z) \in H_k^{r+1,!}$. In particular, since $G_{k,m,r}(z)$ satisfies the conditions for (\ref{halfdep}), $G_{k,m,r}(z)$ has a half-integral depth $r+1/2$.\\

Let $k\leq 0$ and $m >0$. By (\ref{FFourier}), we immediately see that
\begin{align*}
F_{k,-m,0}(z) 
&=(1-k)! m^{\frac{k}{2}} q^{-m}\\ 
&\quad-(1-k) m^{\frac{k}{2}} \Gamma(1-k,4\pi my)q^{-m} - (-4\pi^2 m)^{1-\frac{k}{2}}L_{-m,0}\bigl(1-\frac{k}{2}\bigr)\\
&\quad+ (-1)^{\frac{k}{2}}\sum_{n > 0} 2\pi(1-k)!\sqrt{m/n} L_{-m,n}\bigl(1-\frac{k}{2}\bigr) n^{\frac{k}{2}}q^n\\
&\quad+ (-1)^{\frac{k}{2}}\sum_{n < 0} 2\pi(1-k)\sqrt{|m/n|}L_{-m,n}\bigl(1-\frac{k}{2}\bigr) |n|^{\frac{k}{2}}\Gamma(1-k,4\pi |n|y)q^n.
\end{align*}
Comparing with Duke-Jenkins basis $f_{k,m}(z) = q^{-m} + \sum_{n > \ell_k} a_{k}(m,n)q^n$, for $k<0$ and $m\geq -\ell_k >0$, we see that
\begin{align*}
(1-k)! f_{k,m}(z) - \Biggl\{ m^{-\frac{k}{2}}F_{k,-m,0}(z) + \sum_{\ell_k< n < 0}a_{k}(m,n)|n|^{-\frac{k}{2}}F_{k,n,0}(z)\Biggr\}
\end{align*}
is a harmonic function and bounded on the upper half plane $\mathfrak{H}$. Thus this difference is a constant, that is, equal to 0. For $k=0$ and $m > -\ell_0 = 0$, similarly we have
\begin{align*}
f_{0,m}(z) - F_{0,-m,0}(z) = -4\pi^2mL_{-m,0}(1).
\end{align*}
By Ramanujan \cite{R} it is known that
\begin{align}\label{Ramanujan}
L_{-m,0}(1) = \frac{6}{m\pi^2}\sigma_1(m).
\end{align}
Then it holds that $f_{0,m}(z) = F_{0,-m,0}(z) -24\sigma_1(m)$. In addition, $f_{0,0}(z) = F_{0,0,0}(z) = 1$.\\

As for $k \geq 2$, it is known that the functions $G_{k,m,0}(z)$ with $m>0$ span the space $S_k$ of holomorphic cusp forms. More precisely, Rhoades \cite{Rh} showed the following theorem.
\begin{thm}\cite[Theorem 1.21]{Rh}
Let $k \in 2\mathbb{Z}$ with $k \geq 2$ and $\mathcal{I}$ be a finite set of positive integers. Then
\begin{align*}
\sum_{m \in \mathcal{I}} \overline{\alpha_m} G_{k,m,0}(z) \equiv 0
\end{align*}
if and only if there exists a weakly holomorphic modular form of weight $2-k$ with principal part at $\infty$ equal to
\begin{align*}
\sum_{m \in \mathcal{I}} \frac{\alpha_m}{m^{k/2-1}} q^{-m}.
\end{align*}
\end{thm}
Here this theorem looks different from \cite[Theorem 1.21]{Rh}, because his definition of the Poincar\'{e} series and ours are slightly different. By this theorem and the relation $\ell_{2-k} = -1-\ell_k$, we see that $\{G_{k,m,0}(z)\ |\ 0 < m \leq \ell_k\}$ is a basis for $S_k$. As for $k \geq 2$ and $-m <0$, 
\begin{align*}
G_{k,-m,0}(z) 
&= m^{\frac{k}{2}}q^{-m} + (-1)^{\frac{k}{2}}\sum_{n > 0} 2\pi \sqrt{m/n} L_{-m,n}\bigl(\frac{k}{2}\bigr) n^{\frac{k}{2}}q^n.
\end{align*}
Canceling out the pole at the cusp, we have that
\begin{align}\label{fG}
f_{k,m}(z) - m^{-\frac{k}{2}}G_{k,-m,0}(z)
\end{align}
is a holomorphic cusp form for weight $k > 2$. For $k=2$ and $m>0$, since $f_{2,m}(z) = q^{-m} + \sum_{n=0}^{\infty}a_2(m,n)q^n$ holds, the function (\ref{fG}) is a holomorphic modular form with the constant term $a_2(m,0)$. By $f_{0,0}(z) =1$ and the duality (\ref{dual}), we see $a_2(m,0) = -a_0(0,m) = 0$. Then (\ref{fG}) for $k=2$ is also a holomorphic cusp form. Finally, for $k>2$ and $m=0$ we see that
\begin{align*}
f_{k,0}(z) - E_k(z) = f_{k,0}(z) - \frac{\pi^{\frac{k}{2}}}{\bigl(\frac{k}{2}-1\bigr)k!\zeta(k)}G_{k,0}(z)
\end{align*}
is a holomorphic cusp form. In conclusion, we obtain the following proposition.
\begin{prp}\label{basis}
For an even integer $k \in 2\mathbb{Z}$, we define an integer $\ell_k$ by $k = 12\ell_k + k'$ where $k' \in \{0,4,6,8,10,14\}$. For each integer $m \geq -\ell_k$, the unique weakly holomorphic modular form $f_{k,m}(z) = q^{-m} + \sum_{n > \ell_k}a_k(m,n)q^n$ is expressed in terms of the functions $F_{k,m,0}(z)$, $G_{k,m,0}(z)$ as follows.
\begin{enumerate}
\item For $k \leq -2$, 
\begin{align*}
f_{k,m}(z) = \frac{1}{(1-k)!} \Biggl\{ m^{-\frac{k}{2}}F_{k,-m,0}(z) + \sum_{\ell_k< n < 0}a_k(m,n)|n|^{-\frac{k}{2}}F_{k,n,0}(z)\Biggr\}.
\end{align*}
\item For $k=0$ and $m>0$, $f_{0,m}(z) = F_{0,-m,0}(z) -24\sigma_1(m)$, and $f_{0,0}(z) = F_{0,0,0}(z) =1$.
\item For $k \geq 2$, the set $\{G_{k,m,0}(z)\ |\ 0<m\leq\ell_k\}$ is a basis for the space $S_k$ of holomorphic cusp forms.
\begin{enumerate}
\item For $m>0$, $f_{k,m}(z) -m^{-\frac{k}{2}}G_{k,-m,0}(z)$ is a holomorphic cusp form. 
\item For $m=0$, $f_{k,0}(z) - \pi^{\frac{k}{2}}\bigl\{\bigl(\frac{k}{2}-1\bigr)k!\zeta(k)\bigr\}^{-1}G_{k,0,0}(z)$ is a holomorphic cusp form.
\item For $m<0$, $f_{k,m}(z)$ is a holomorphic cusp form.
\end{enumerate}
\end{enumerate}
For all $k \in 2\mathbb{Z}$, the set $\{f_{k,m}(z)\ |\ m \geq -\ell_k\}$ is a basis for $M_k^!$.
\end{prp}

\section{Proof of Theorem \ref{Main2}}\label{s4}
We consider the Taylor expansion form of $P_{k,m}(z,s)$ again,
\begin{align*}
P_{k,m}(z,s) = \left\{\begin{array}{ll}
\sum_{r=0}^{\infty} F_{k,m,r}(z)\bigl(s+\frac{k}{2}-1\bigr)^r \quad &\text{if } k \leq 0, \\
\ \\
\sum_{r=0}^{\infty} G_{k,m,r}(z)\bigl(s-\frac{k}{2}\bigr)^r \quad &\text{if } k \geq 2.
\end{array} \right.
\end{align*}
Let $k \leq 0$, we have
\begin{align*}
\xi_k P_{k,m}(z,s) = \sum_{r=0}^{\infty}\xi_k F_{k,m,r}(z)\bigl(\bar{s}+\frac{k}{2}-1\bigr)^r.
\end{align*}
On the other hand, by (\ref{xiP}), 
\begin{align*}
\xi_k P_{k,m}(z,s) &= (4\pi)^{1-k}\bigl(\bar{s}-\frac{k}{2}\bigr)P_{2-k,-m}(z,\bar{s})\\
&= (4\pi)^{1-k}\bigl(1-k+\bar{s}+\frac{k}{2}-1\bigr)\sum_{r=0}^{\infty} G_{2-k,-m,r}(z)\bigl(\bar{s}+\frac{k}{2}-1\bigr)^r\\
&= (4\pi)^{1-k}\Biggl\{(1-k)\sum_{r=0}^{\infty} G_{2-k,-m,r}(z)\bigl(\bar{s}+\frac{k}{2}-1\bigr)^r + \sum_{r=0}^{\infty} G_{2-k,-m,r}(z)\bigl(\bar{s}+\frac{k}{2}-1\bigr)^{r+1}\Biggr\}.
\end{align*}
Term by term comparison in $\bigl(\bar{s}+\frac{k}{2}-1\bigr)^r$ yields
\begin{align*}
\xi_kF_{k,m,r}(z) = (4\pi)^{1-k}\biggl\{(1-k)G_{2-k,-m,r}(z) + G_{2-k,-m,r-1}(z)\biggr\}.
\end{align*}

Let $k \geq 2$, we have
\begin{align*}
\xi_kP_{k,m}(z,s) = \sum_{r=0}^{\infty} \xi_kG_{k,m,r}(z)\bigl(\bar{s}-\frac{k}{2}\bigr)^r.
\end{align*}
Similarly, 
\begin{align*}
\xi_kP_{k,m}(z,s) &= (4\pi)^{1-k}\bigl(\bar{s}-\frac{k}{2}\bigr)P_{2-k,-m}(z,\bar{s})\\
&= (4\pi)^{1-k}\bigl(\bar{s}-\frac{k}{2}\bigr) \sum_{r=0}^{\infty} F_{2-k,-m,r}(z)\bigl(\bar{s}-\frac{k}{2}\bigr)^r.
\end{align*}
Then we have
\begin{align*}
\xi_kG_{k,m,r}(z) = (4\pi)^{1-k}F_{2-k,-m,r-1}(z).
\end{align*}

\section{Proof of Theorem \ref{Main1}}\label{s5}
By the Fourier-Whittaker expansion form of $F_{k,m,r}(z)$ and $G_{k,m,r}(z)$ in Section \ref{s3.5}, it can be easily checked that, in the description of Proposition \ref{FWE}, the Fourier-Whittaker coefficients of $u_{k,m}^{[r],+}(y)e^{2\pi imx}$ satisfy
\begin{align*}
r!(-4\pi)^{\frac{k}{2}}c_{m,r}^+ &= \frac{(-1)^r \Gamma(2s)}{\Gamma(s-\mathrm{sgn}(m)k/2)}\bigg|_{s=1-\frac{k}{2}} = \frac{(-1)^r \Gamma(2-k)}{\Gamma(1-\frac{k}{2}-\mathrm{sgn}(m)k/2)} \neq 0,\quad \text{for}\ k \leq 0,\\
r!(-4\pi)^{\frac{k}{2}}c_{m,r}^+ &= \frac{\Gamma(2s)}{\Gamma(s-\mathrm{sgn}(m)k/2)}\bigg|_{s=\frac{k}{2}} = \frac{\Gamma(k)}{\Gamma(\frac{k}{2}-\mathrm{sgn}(m)k/2)} = 1 \neq 0,\quad \text{for}\ k\geq 2,\ m < 0.
\end{align*}
For the case of $k \geq 2$ and $m>0$, since $\Gamma(2s)/\Gamma(s-\mathrm{sgn}(m)k/2)$ has a simple zero at $s=k/2$, we have
\begin{align*}
r!(-4\pi)^{\frac{k}{2}}c_{m,r}^+ &= \biggl( \frac{\Gamma(2s)}{\Gamma(s-\mathrm{sgn}(m)k/2)}\biggr)\bigg|_{s=\frac{k}{2}} = 0,\\
r!(-4\pi)^{\frac{k}{2}}c_{m,r-1}^+ &= r\cdot \frac{\partial}{\partial s} \biggl( \frac{\Gamma(2s)}{\Gamma(s-\mathrm{sgn}(m)k/2)}\biggr)\bigg|_{s=\frac{k}{2}} \neq 0.
\end{align*}
Thus for an even integer $k \leq 0$ and an arbitrary $f(z) \in H_k^{r,!}$, canceling out all exponentially growing terms $u_{k,n}^{[j],+}(y)$, we have
\begin{align*}
f(z) - \sum_{j=0}^{r-1}\sum_{m\in S} a_{k,m,j}^{} F_{k,m,j}(z) \in H_k^r,
\end{align*}
where $S \subset \mathbb{Z}$ is a finite subset and $a_{k,m,j} \in \mathbb{C}$. By Theorem \ref{LRthm}, the set $\{F_{k,m,r-1}(z)\ |\ m \in \mathbb{Z}\}$ is a generating set of $H_k^{r,!}/H_k^{r-1,!}$, where we put $F_{k,0,r-1}(z) := F_{k,r-1}(z)$ in Theorem \ref{LRthm}. Next we show that the subset $\{F_{k,m,r-1}(z)\ |\ m \geq -\ell_{2-k}\}$ is a basis of $H_k^{r,!}/H_k^{r-1/2,!}$. For any finite subset $S \subset \mathbb{Z}_{\geq -\ell_{2-k}}$, we assume that
\begin{align*}
\sum_{m\in S} a_{k,m,j}^{} F_{k,m,r-1}(z) =0.
\end{align*}
Taking the action of $\xi_k \circ \Delta_k^{r-1}$, we have
\begin{align}\label{xidelta}
\xi_k \circ \Delta_k^{r-1} \Biggl(\sum_{m\in S} a_{k,m,r-1}^{} F_{k,m,r-1}(z)\Biggr) &= \sum_{m \in S}b_{k,m,r-1}G_{2-k,-m,0}(z),
\end{align}
where
\begin{align*}
b_{k,m,r-1} = \left\{\begin{array}{ll}
-(k-1)^r(4\pi)^{1-k}a_{k,m,r-1}^{} \quad &\text{if } m \neq 0, \\
(k-1)^{r-1}a_{k,0,r-1}^{} \quad &\text{if } m=0.
\end{array} \right.
\end{align*}
The result (\ref{xidelta}) is a weakly holomorphic modular form of weight $2-k$, and by assumption it is equal to $0$. By Proposition \ref{basis}, we can see that the set $\{G_{2-k,-m,0}(z)\ |\ m \geq -\ell_{2-k}\}$ is a basis for $M_{2-k}^!$. Then each coefficient $b_{k,m,r-1}=0$, that is, $a_{k,m,r-1}^{} =0$. Thus $\{F_{k,m,r-1}(z)\ |\ m \geq -\ell_{2-k}\}$ is linearly independent. On the other hand, for $m < -\ell_{2-k}$, there exists a finite subset $S \subset \mathbb{Z}_{\geq -\ell_{2-k}}$ and $a_{k,n,r-1} \in \mathbb{C}$ such that
\begin{align*}
\xi_k \circ \Delta_k^{r-1} \Biggl(F_{k,m,r-1}(z) - \sum_{n \in S}a_{k,n,r-1}^{}F_{k,n,r-1}(z)\Biggr) = 0,
\end{align*}
that is, $F_{k,m,r-1}(z)$ is written as a linear combination of $F_{k,n,r-1}(z)$ with $n \geq -\ell_{2-k}$ in $H_k^{r,!}/H_k^{r-1/2,!}$. Thus for $k \leq 0$, we conclude that $\{F_{k,m,r-1}(z)\ |\ m \geq -\ell_{2-k}\}$ is a basis of $H_k^{r,!}/H_k^{r-1/2,!}$. Since it holds that $\ell_{2-k}=-1-\ell_k$, we obtain Theorem \ref{Main1} (1-a). By the same calculation as (\ref{xidelta}) and Proposition \ref{basis}, 
\begin{align*}
\Biggl\{ \tilde{F}_{k,m,r-1}(z) := |m|^{-\frac{k}{2}}F_{k,m,r-1}(z) + \sum_{\ell_k<n<0}a_k(-m,n)|n|^{-\frac{k}{2}}F_{k,n,r-1}(z)\ \Biggr|\ m\leq \ell_k\Biggr\}
\end{align*}
is a basis of $H_k^{r-1/2,!}/H_k^{r-1,!}$. The key property is for $k \leq -2$ and $m\geq -\ell_k>0$,
\begin{align*}
\Delta_k^{r-1}\tilde{F}_{k,-m,r-1}(z) &= (k-1)^{r-1} \Biggl( m^{-\frac{k}{2}}F_{k,-m,0}(z) + \sum_{\ell_k<n<0}a_k(m,n)|n|^{-\frac{k}{2}}F_{k,n,0}(z)\Biggr)\\
&= (k-1)^{r-1}(1-k)!f_{k,m}(z).
\end{align*}
As for $k=0$ and $m \geq -\ell_0 = 0$, 
\begin{align*}
\Delta_0^{r-1}\tilde{F}_{0,-m,r-1}(z) &= (-1)^{r-1}F_{0,-m,0}(z)\\
&= (-1)^{r-1} \left\{\begin{array}{ll}
f_{0,m}(z) + 24\sigma_1(m) \quad &\text{if } m > 0, \\
f_{0,0}(z) \quad &\text{if } m = 0.
\end{array} \right.
\end{align*}
For an arbitrary $f \in H_k^{r-1/2,!}$, it holds that $\Delta_k^{r-1}f \in H_k^{1/2,!} = M_k^!$ by Proposition \ref{xilift}. Then $f$ can be expressed as a linear combination of $\tilde{F}_{k,m,r-1}(z)$ with $m \leq \ell_k$ in $H_k^{r-1.2,!}/H_k^{r-1,!}$. By these properties, we obtain Theorem \ref{Main1} (1-b).\\

As for the case of $k \geq 2$, we see that the set $\{ G_{k,m,r}(z)\ |\ m\in \mathbb{Z}\}$ is a generating set of $H_k^{r+1/2,!}/H_k^{r-1/2,!}$. Similarly, we can take $\{G_{k,m,r}(z)\ |\ m \leq \ell_k\}$ as a basis of $H_k^{r+1/2,!}/H_k^{r,!}$. Moreover, by Proposition \ref{basis},
\begin{align*}
\Biggl\{ \tilde{G}_{k,m,r}(z) := m^{\frac{k}{2}-1}G_{k,m,r}(z) - \sum_{0<n \leq \ell_k}a_k(-n,m)n^{\frac{k}{2}-1}G_{k,n,r}(z)\ \Biggr|\ m > \ell_k\Biggr\}
\end{align*}
is a basis of $H_k^{r,!}/H_k^{r-1/2,!}$. The key properties are 
\begin{align*}
\xi_k \circ \Delta_k^{r-1} \tilde{G}_{k,m,r}(z) &= (4\pi)^{1-k}(1-k)^{r-1}(k-1)! f_{2-k,m}(z) &\text{for } m>0, k\geq 2,\\
\xi_2 \circ \Delta_2^{r-1} \tilde{G}_{2,0,r}(z) &= (-1)^{r-1}F_{0,0,0}(z) = (-1)^{r-1}f_{0,0}(z) &\text{for } m=0, k=2,
\end{align*}
and $m>\ell_k \iff m \geq -\ell_{2-k}$. This concludes the proof of Theorem \ref{Main1}.

\section{Examples}\label{s6}
In this last section, we follow the expositions of some examples of the functions $F_{k,m,r}(z)$, $G_{k,m,r}(z)$ related to the previous works by \cite{DIT2}, \cite{JKK}, and \cite{O}.\\

(1) In the cases of $k=12$ and $-10$. Ono \cite{O} constructed a mock modular form $M_{\Delta}(z)$ for the discriminant function $\Delta(z)$ of the form
\begin{align*}
\frac{1}{11!} \cdot M_{\Delta}(z) = \frac{1}{q} + \frac{24}{B_{12}} + O(q),
\end{align*}
where $B_{12} = -691/2730$ is the 12th Bernoulli number. This function satisfies the property that
\begin{align*}
R_{\Delta}(z) := M_{\Delta}(z) + (2\pi)^{11} \cdot 11i \cdot \beta_{\Delta} \int_{-\bar{z}}^{i \infty} \frac{\overline{\Delta(-\bar{\tau})}}{(-i(\tau+z))^{-10}}d\tau
\end{align*}
is a weight $-10$ harmonic weak Maass form on $\mathrm{SL}_2(\mathbb{Z})$, where the constant $\beta_{\Delta} = 2.840287\dots$ is defined by $G_{12,1,0}(z) = \beta_{\Delta} \Delta(z)$. Then this function is expressed in our notations as $F_{-10,-1,0}(z) = R_{\Delta}(z)$.\\

(2) In the case of $k=2$. Duke-Imamo\={g}lu-T\'{o}th \cite{DIT2} gave the explicit formulas
\begin{align*}
G_{2,m,0}(z) = P_{2,m}(z,1) = \left\{\begin{array}{ll}
0 \quad &\text{if } m > 0, \\
-mq^m - \sum_{n>0} n c_{m}(n) q^n = -j'_m(z) \quad &\text{if } m < 0,
\end{array} \right.
\end{align*}
where $j_m \in \mathbb{C}[j]$ is the unique polynomial in the elliptic modular $j$-function having a Fourier expansion of the form
\begin{align*}
j_m(z) = q^m + \sum_{n=1}^{\infty}c_{m}(n)q^n.
\end{align*} 
Moreover they showed that the space $H_2^{1,!}$ is spanned by $\{G_{2,m,0}(z)\ |\ m<0\} \cup \{ E_2^*(z) \} \cup \{G_{2,m,1}(z)\ |\ m>0\}$. Furthermore they gave the Fourier-Whittaker expansion of $G_{2,m,1}(z)$ in terms of the regularized Petersson inner products.\\

(3) In the case of $k=0$. For a positive integer $m$, Jeon-Kang-Kim \cite{JKK} considered a polyharmonic Maass form $\widehat{J}_m(z)$ satisfying $\xi_k \widehat{J}_m(z) = 4\pi G_{2,m,1}(z)$ and $\Delta_0 \widehat{J}_m(z) = -j_{-m}(z) -24\sigma_1(m)$, and its cycle integral. According to this paper, their $\widehat{J}_m(z)$ coincides our $F_{0,-m,1}(z)$, and we get $F_{0,-m,0}(z) = j_{-m}(z) + 24\sigma_1(m)$ from (\ref{Ramanujan}), originated with Ramanujan \cite{R}.

\end{document}